\documentclass{article}

\usepackage{latexsym,amsmath,amsfonts,amsthm,amssymb,graphicx}

\numberwithin{equation}{section}

\newtheorem{Th}{Theorem}[section]

\theoremstyle{definition}
\newtheorem{Example}[Th]{Example}
\newtheorem{df}[Th]{Definition}

\theoremstyle{remark}
\newtheorem{Remark}[Th]{Remark}

\newcommand{\R}{\mathbb{R}}

\newcommand{\N}{\mathbb{N}}
\newcommand{\T}{\mathbb{T}}

\begin{document}

\title{STRONG MINIMIZERS OF THE CALCULUS OF VARIATIONS
ON TIME SCALES\\  AND THE WEIERSTRASS CONDITION\thanks{Accepted 12/May/2009 for publication in the
Proceedings of the Estonian Academy of Sciences.}}

\author{Agnieszka B. Malinowska\footnote{On leave of absence from Faculty of Computer Science, Bia{\l}ystok Technical University,
15-351 Bia\l ystok, Poland. E-mail: abmalina@pb.bialystok.pl}\\
\texttt{abmalinowska@ua.pt}
\and
Delfim F. M. Torres\\
\texttt{delfim@ua.pt}}

\date{Department of Mathematics\\
      University of Aveiro\\
      3810-193 Aveiro, Portugal}

\maketitle


\begin{abstract}
We introduce the notion of strong local minimizer
for the problems of the calculus of variations
on time scales. Simple examples show that on a time scale
a weak minimum is not necessarily a strong minimum.
A time scale form of the Weierstrass necessary optimality condition is proved, which enables to include and generalize
in the same result both continuous-time
and discrete-time conditions.
\end{abstract}

\smallskip


\smallskip


\smallskip



\section{Introduction}

Dynamic equations on time scales
is a recent subject that allows the unification and extension
of the study of differential and difference equations in one and same theory \cite{BP}.

The calculus of variations on time scales was introduced in 2004
with the papers of Martin Bohner \cite{b7} and Roman Hilscher and
Vera Zeidan \cite{HZ1}. Roughly speaking, in \cite{b7} the basic
problem of the calculus of variations on time scales with given
boundary conditions is introduced, and time scale versions of the
classical necessary optimality conditions of Euler-Lagrange and
Legendre proved, while in \cite{HZ1} necessary conditions as well as
sufficient conditions for variable end-points calculus of variations
problems on time scales are established. Since the two pioneer works
\cite{b7,HZ1} and the understanding that much remains to be done in
the area \cite{F:T:07}, several recent studies have been dedicated
to the calculus of variations on time scales: the time scale
Euler-Lagrange equation was proved for problems with double
delta-integrals \cite{B:G:07} and for problems with higher-order
delta-derivatives \cite{F:T:08}; a correspondence between the
existence of variational symmetries and the existence of conserved
quantities along the respective Euler-Lagrange delta-extremals was
established in \cite{b:t:08}; optimality conditions for
isoperimetric problems on time scales with multiple constraints and
Pareto optimality conditions for multiobjective delta variational
problems were studied in \cite{abmalina:t}; a weak maximum principle
for optimal control problems on time scales has been obtained in
\cite{HZ2}. Such results may also be formulated via the
nabla-calculus on time scales, and seem to have interesting
applications in economics
\cite{Almeida:T,Atici:Uysal:06,Atici:Uysal:08,NataliaHigherOrderNabla}.

In all the works available in the literature on time scales the
variational extrema are regarded in a weak local sense. Differently,
here we consider strong solutions of problems of the calculus of
variations on time scales. In Section~\ref{sec:tsc} we briefly
review the necessary results of the calculus on time scales. The
reader interested in the theory of time scales is referred to
\cite{BP, BP2}, while for the classical continuous-time calculus of
variations we refer to \cite{BM,L}, and to \cite{K:P:91} for the
discrete-time setting. In Section~\ref{sec:mr} the concept of strong
local minimum is introduced (\textrm{cf.} Definition~\ref{def:slm}),
and an example of a problem of the calculus of variations on the
time scale $\mathbb{T}=\{\frac{1}{n}:n\in \N\}\cup \{0 \}$ is
considered showing that the standard weak minimum used in the
literature on time scales is not necessarily a strong minimum
(\textrm{cf.} Example~\ref{ex:weak:vs:strong}). Our main result is a
time scale version of the Weierstrass necessary optimality condition
for strong local minimum (\textrm{cf.} Theorem~\ref{Weier}). We end
with Section~\ref{sec:ex}, illustrating our main result with the
particular cases of discrete-time and the $q$-calculus of variations
\cite{Bangerezako}.


\section{Time Scales Calculus}
\label{sec:tsc}

In this section we introduce basic definitions and results that will
be needed for the rest of the paper. For a more general theory of
calculus on time scales, we refer the reader to \cite{BP, BP2}.

A nonempty closed subset of $\mathbb{R}$ is called a \emph{time
scale} and it is denoted by $\mathbb{T}$. Thus, $\mathbb{R}$,
$\mathbb{Z}$, and $\mathbb{N}$, are trivial examples of times
scales. Other examples of times scales are: $[-2,4] \bigcup
\mathbb{N}$, $h\mathbb{Z}:=\{h z | z \in \mathbb{Z}\}$ for some
$h>0$, $q^{\mathbb{N}_0}:=\{q^k | k \in \mathbb{N}_0\}$ for some
$q>1$, and the Cantor set. We assume that a time scale $\mathbb{T}$
has the topology that it inherits from the real numbers with the
standard topology.

The \emph{forward jump operator}
$\sigma:\mathbb{T}\rightarrow\mathbb{T}$ is defined by
$$\sigma(t)=\inf{\{s\in\mathbb{T}:s>t}\},
\mbox{ for all $t\in\mathbb{T}$},$$ while the \emph{backward jump
operator} $\rho:\mathbb{T}\rightarrow\mathbb{T}$ is defined by
$$\rho(t)=\sup{\{s\in\mathbb{T}:s<t}\},\mbox{ for all
$t\in\mathbb{T}$},$$ with $\inf\emptyset=\sup\mathbb{T}$
(\textrm{i.e.}, $\sigma(M)=M$ if $\mathbb{T}$ has a maximum $M$) and
$\sup\emptyset=\inf\mathbb{T}$ (\textrm{i.e.},
$\rho(m)=m$ if $\mathbb{T}$ has a minimum $m$).

If $\sigma(t)>t$, we say that $t$ is \emph{right-scattered}, while
if $\rho(t)<t$ we say that $t$ is \emph{left-scattered}. Also, if
$t<\sup\mathbb{T}$ and $\sigma(t)=t$, than $t$ is called
\emph{right-dense}, and if $t
>\inf\mathbb{T}$ and $\rho(t)=t$, then $t$ is called
\emph{left-dense}. The set $\mathbb{T}^{\kappa}$ is defined as
$\mathbb{T}$ without the left-scattered maximum of $\mathbb{T}$ (in
case it exists).

The \emph{graininess function} $\mu:\mathbb{T}\rightarrow[0,\infty)$
is defined by
$$\mu(t)=\sigma(t)-t,\mbox{ for all $t\in\mathbb{T}$}.$$

\begin{Example}
If $\mathbb{T} = \mathbb{R}$, then $\sigma(t) = \rho(t) = t$ and
$\mu(t)= 0$. If $\mathbb{T} = \mathbb{Z}$, then $\sigma(t) = t + 1$,
$\rho(t) = t - 1$, and $\mu(t)= 1$. On the other hand, if
$\mathbb{T} = q^{\mathbb{N}_0}$, where $q>1$ is a fixed real number,
then we have $\sigma(t) = q t$, $\rho(t) = q^{-1} t$, and $\mu(t)=
(q-1) t$.
 \end{Example}

A function $f:\mathbb{T}\rightarrow\mathbb{R}$ is \emph{regulated}
if the right-hand limit $f(t+)$ exists (finite) at all right-dense
points $t\in \mathbb{T}$ and the left-hand limit $f(t-)$ exists at
all left-dense points $t\in \mathbb{T}$. A function $f$ is
\emph{rd-continuous} (we write $f\in C_{rd}$) if it is regulated and
if it is continuous at all right-dense points $t\in \mathbb{T}$.
Following \cite{HZ1}, a function $f$ is \emph{ piecewise
rd-continuous} (we write $f\in C_{prd}$) if it is regulated and if
it is rd-continuous at all, except possibly at finitely many,
right-dense points $t\in \mathbb{T}$.

We say that a function $f:\mathbb{T}\rightarrow\mathbb{R}$ is
\emph{delta differentiable} at $t\in\mathbb{T}^{\kappa}$ if there
exists a number $f^{\triangle}(t)$ such that for all $\varepsilon>0$ there is a neighborhood $U$ of $t$ (\textrm{i.e.},
$U=(t-\delta,t+\delta)\cap\mathbb{T}$ for some $\delta>0$) such that
$$|f(\sigma(t))-f(s)-f^{\Delta}(t)(\sigma(t)-s)|
\leq\varepsilon|\sigma(t)-s|,\mbox{ for all $s\in U$}.$$ We call
$f^{\triangle}(t)$ the \emph{delta derivative} of $f$ at $t$ and say
that $f$ is \emph{delta differentiable} on $\mathbb{T}^{\kappa}$
provided $f^{\triangle}(t)$ exists for all
$t\in\mathbb{T}^{\kappa}$. Note that in right-dense points
$f^{\triangle} (t)=\lim_{s\rightarrow t}=\frac{f(t)-f(s)}{t-s}$,
provided this limit exists, and in right-scattered points
$f^{\triangle} (t)=\frac{f(\sigma(t)-f(t)}{\mu(t)}$ provided $f$ is
continuous at $t$.

\begin{Example}
\label{ex:der:QC} If $\mathbb{T} = \mathbb{R}$, then $f^{\Delta}(t)
= f'(t)$, \textrm{i.e.}, the delta derivative coincides with the
usual one. If $\mathbb{T} = \mathbb{Z}$, then $f^{\Delta}(t) =
\Delta f(t) = f(t+1) - f(t)$. If $\mathbb{T} = q^{\mathbb{N}_0}$,
$q>1$, then $f^{\Delta} (t)=\frac{f(q t)-f(t)}{(q-1) t}$,
\textrm{i.e.}, we get the usual derivative of Quantum calculus
\cite{QC}.
\end{Example}

Let $f, g: \mathbb{T} \rightarrow \mathbb{R}$ be delta
differentiable at $t \in \T^{\kappa}$. Then (see, e.g., \cite{BP}),

\begin{itemize}

\item[(i)] the product $fg$ is delta differentiable at $t$ with
\begin{equation*}
(fg)^{\triangle}(t)=f^{\triangle}(t)g^\sigma(t)+f(t)g^{\triangle}(t)=
f^{\triangle}(t)g(t)+f^{\sigma}(t)g^{\triangle}(t) \, ;
\end{equation*}

\item[(ii)]if $g(t)g^{\sigma}(t)\neq 0$, then $\frac{f}{g}$ is delta differentiable at $t$ with
\begin{equation*}
\left(\frac{f}{g}\right)^{\triangle}(t)=\frac{f^{\triangle}(t)g(t)
-f(t)g^{\triangle}(t)}{g(t)g^{\sigma}(t)},
\end{equation*}
where we abbreviate here and throughout the text $f\circ\sigma$ by
$f^\sigma$.
\end{itemize}

A function $f$ is \emph{ rd-continuously delta differentiable} (we
write $f\in C_{rd}^{1}$) if $f^{\triangle}$ exists for all $t \in
\T^{\kappa}$ and $f^{\triangle}\in C_{rd}$. A continuous function
$f$ is  \emph{ piecewise rd-continuously delta differentiable} (we
write $f\in C_{prd}^{1}$) if $f$ is continuous and $f^{\triangle}$
exists for all, except possibly at finitely many, $t \in
\T^{\kappa}$ and $f^{\triangle}\in C_{rd}$. It is known that
piecewise rd-continuous functions possess an \emph{antiderivative},
\textrm{i.e.}, there exists a function $F$ with $F^{\triangle}=f$,
and in this case the \emph{delta integral} is defined by
$\int_{c}^{d}f(t)\triangle t=F(d)-F(c)$ for all $c,d\in\T$.

\begin{Example}
Let $a, b \in \mathbb{T}$ with $a < b$. If $\mathbb{T} =
\mathbb{R}$, then $\int_{a}^{b}f(t)\Delta t = \int_{a}^{b}f(t) dt$,
where the integral on the right-hand side is the classical Riemann
integral. If $\mathbb{T} = \mathbb{Z}$, then $\int_{a}^{b}f(t)\Delta
t = \sum_{k=a}^{b-1} f(k)$. If $\mathbb{T} = q^{\mathbb{N}_0}$,
$q>1$, then $\int_{a}^{b}f(t)\Delta t = (1 - q) \sum_{t \in [a,b)} t
f(t)$.
\end{Example}

The delta integral has the following properties (see, e.g.,
\cite{BP}):
\begin{itemize}

\item[(i)] if $f\in C_{prd}$ and $t \in \T^{\kappa}$, then

\begin{equation*}
\int_t^{\sigma(t)}f(\tau)\triangle\tau=\mu(t)f(t) \, ;
\end{equation*}

\item[(ii)]if $c,d\in\T$ and $f,g\in C_{prd}$, then

\begin{equation*}
\int_{c}^{d}f(\sigma(t))g^{\triangle}(t)\triangle t
=\left[(fg)(t)\right]_{t=c}^{t=d}-\int_{c}^{d}f^{\triangle}(t)g(t)\triangle
t;
\end{equation*}

\begin{equation*}
\int_{c}^{d}f(t)g^{\triangle}(t)\triangle t
=\left[(fg)(t)\right]_{t=c}^{t=d}-\int_{c}^{d}
f^{\triangle}(t)g(\sigma(t))\triangle t.
\end{equation*}
\end{itemize}


\section{The Weierstrass Necessary Condition}
\label{sec:mr}

Let $\T$ be a bounded time scale. Throughout we let $t_{0},t_{1}\in
\T$ with $t_{0}< t_{1}$. For an interval $[t_{0},t_{1}]\cap \T$ we simply write $[t_{0},t_{1}]$. The problem of the calculus
of variations on time scales under consideration has the form
\begin{equation}\label{vp}
 \text{minimize} \quad   \mathcal{L}[x]=\int_{t_{0}}^{t_{1}}f(t,x^{\sigma}(t),x^{\triangle}(t))\triangle
     t ,
\end{equation}
over all $x\in C_{prd}^{1}$ satisfying the boundary conditions
\begin{equation}\label{bc}
    x(t_{0})=\alpha,\,  \ x(t_{1})=\beta, \quad \alpha,\beta\in \R,
\end{equation}
where $f:[t_{0}, t_{1}]^{\kappa}\times \R \times \R \rightarrow \R$.

A function $x\in C_{prd}^{1}$ is said to be admissible if it
satisfies conditions \eqref{bc}.

Let us consider two norms in $C_{prd}^{1}$:
\begin{equation*}
    \|x\|_{1}=\sup_{t\in[t_{0},t_{1}]^{\kappa}}|x^{\sigma}(t)|
    +\sup_{t\in[t_{0},t_{1}]^{\kappa}\setminus T}|x^{\triangle}(t)|,
\end{equation*}
where here and subsequently $T$ denotes the set of points of
$[t_{0},t_{1}]^{\kappa}$ where $x^{\triangle}(t)$ does not exist,
and
\begin{equation*}
    \|x\|_{0}=\sup_{t\in[t_{0},t_{1}]^{\kappa}}|x^{\sigma}(t)|.
\end{equation*}
The norms $\|\cdot \|_{0}$ and $\|\cdot \|_{1}$ are called the \emph{strong} and the \emph{weak} norm, respectively. The
strong and weak norms lead to the following definitions for local
minimum:

\begin{df}
\label{def:slm} An admissible function $\bar{x}$ is said to be a
\emph{strong local minimum } for \eqref{vp}--\eqref{bc} if there
exists $\delta >0$ such that $\mathcal{L}[\bar{x}]\leq
\mathcal{L}[x]$ for all admissible $x$ with
$\|x-\bar{x}\|_{0}<\delta$. Likewise, an admissible function
$\bar{x}$ is said to be a \emph{weak local minimum} for
\eqref{vp}--\eqref{bc} if there exists $\delta >0$ such that
$\mathcal{L}[\bar{x}]\leq \mathcal{L}[x]$ for all admissible $x$
with $\|x-\bar{x}\|_{1}<\delta$.
\end{df}

A weak minimum may not necessarily be a strong minimum:

\begin{Example}
\label{ex:weak:vs:strong}
Consider the variational problem
\begin{equation}\label{pr}
    \mathcal{L}[x]=\int_{0}^{1}[x^{\triangle}(t)^{2}
    -x^{\triangle}(t)^{4}]\triangle
    t,   \quad x(0)=0\, , \quad x(1)=0 \, ,
\end{equation}
on the time scale $\mathbb{T}=\{\frac{1}{n}:n\in \N\}\cup \{0 \}$
(note that we need to add zero in order to have a closed set). Let
us show that $\tilde{x}(t)=0$, $0\leq t \leq 1$ is a weak local
minimum for \eqref{pr}. In the topology induced by $\|\cdot\|_{1}$
consider the open ball of radius $1$ centered at $\tilde{x}$,
\textrm{i.e.},
\begin{equation*}
B_{1}^{1}(\tilde{x})=\left\{x\in
C_{prd}^{1}:\|x-\tilde{x}\|_{1}<1\right \}.
\end{equation*}
We use the notation $B_{r}^{k}$ for the ball of radius $r$ in norm
$\|\cdot\|_{k}$, $k=1,2$. For every $x\in B_{1}^{1}(\tilde{x})$ we
have
\begin{equation*}
    |x^{\triangle}(t)|\leq 1,   \quad \forall t\in [0,1]^{\kappa},
\end{equation*}
hence $\mathcal{L}[x]\geq 0$. This proves that $\tilde{x}$ is a weak local minimum
for \eqref{pr} since $\mathcal{L}[\tilde{x}]=0$.
Now let us consider the function defined by
\begin{gather*}
x_{d}(t)=
\begin{cases}
d & \text{ if } t=\sigma(t_{0}) \\
 0 & \text{ otherwise }
\end{cases} \, ,
\quad t_{0}\in (0,1)\cap \mathbb{T}, \quad \sigma(t_{0})\neq 1\,
,\quad d\in \R\backslash\{0 \} .
\end{gather*}
Function $x_{d}$ is admissible and
 $\|x_{d}\|_{0}=\sup_{t\in[0,1]^{\kappa}}|x^{\sigma}_{d}(t)|=|d|$.
Therefore, for every $\delta >0$ there is a $d$ such that
\begin{equation*}
x_{d}\in B_{\delta}^{0}(\tilde{x})=\left\{x\in
C_{prd}^{1}:\|x-\tilde{x}\|_{0}<\delta\right \}.
\end{equation*}
We have
\begin{equation*}
x_{d}^{\triangle}(t_{0})=\frac{d}{\mu (t_{0})},
\end{equation*}
\begin{equation*}
x_{d}^{\triangle}(\sigma (t_{0}))=\frac{-d}{\mu (\sigma (t_{0}))} \, ,
\end{equation*}
and $x_{d}^{\triangle}(t)=0$ for all $t\neq t_{0},\sigma (t_{0})$.
Hence, $|x_{d}^{\triangle}(t)|$, $0\leq t \leq 1$, can take arbitrary
large values since $\mu (t)=\frac{t^{2}}{1-t}\rightarrow 0$ as
$t\rightarrow 0$.
Note that for every $\delta >0$ we can choose $d$ and $t_{0}$ such
that $x_{d}\in B_{\delta}^{0}(\tilde{x})$ and $\frac{d}{\mu (\sigma
(t_{0}))}> 1$.
Finally,
\begin{equation*}
\begin{split}
\mathcal{L}[x_{d}]&=\int_{0}^{1}[x_{d}^{\triangle}(t)^{2}
-x^{\triangle}_{d}(t)^{4}]\triangle t \\
&=\mu (t_{0})\left [(\frac{d}{\mu (t_{0})})^{2}
-(\frac{d}{\mu
    (t_{0})})^{4}\right ]+ \mu (\sigma(t_{0}))\left [(\frac{d}{\mu (\sigma(t_{0}))})^{2}
-(\frac{d}{\mu
    (\sigma(t_{0}))})^{4}\right]\\
&=\frac{d^{2}}{\mu (t_{0})}\left [1-\frac{d^{2}}{\mu ^{2}
    (t_{0})}\right]+\frac{d^{2}}{\mu (\sigma (t_{0}))}\left[1-\frac{d^{2}}{\mu ^{2}
    (\sigma(t_{0}))}\right]<0 \, .
\end{split}
\end{equation*}
Therefore, the trajectory $\tilde{x}$ cannot be a strong minimum for \eqref{pr}.
\end{Example}

\bigskip

From now on we assume that $f:[t_{0}, t_{1}]^{\kappa}\times \R
\times \R \rightarrow \R$ has partial continuous derivatives $f_{x}$
and $f_{v}$, respectively with respect to the second and third
variables, for all $t\in[t_{0}, t_{1}]^{\kappa}$, and
$f(\cdot,x,v)$, $f_{x}(\cdot,x,v)$ and $f_{v}(\cdot,x,v)$ are
continuous.

Let $E:[t_{0}, t_{1}]^{\kappa}\times \R^{3} \rightarrow \R$ be the
function with the values
\begin{equation*}
E(t,x,r,q)=f(t,x,q)-f(t,x,r)-(q-r)f_{r}(t,x,r) \, .
\end{equation*}
This function, called the \emph{Weierstrass excess function}, is
utilized in the following theorem:

\begin{Th}[Weierstrass necessary optimality condition on time scales]
\label{Weier}
Let $\T$ be a time scale, $t_{0},t_{1}\in \T$, $t_{0}< t_{1}$.
Assume that the function $f(t,x,r)$ in problem \eqref{vp}--\eqref{bc}
satisfies the following condition:
\begin{equation}\label{convex}
    \mu(t)f(t,x,\gamma r_{1}+(1-\gamma)r_{2})\leq \mu(t)\gamma f(t,x,
    r_{1})+\mu(t)(1-\gamma)f(t,x,r_{2})
\end{equation}
for each $(t,x)\in[t_{0}, t_{1}]^{\kappa}\times \R$, all
$r_{1},r_{2}\in\R$ and $\gamma \in [0,1]$. Let $\bar{x}$ be a
piecewise continuous function. If $\bar{x}$ is a strong local
minimum for \eqref{vp}--\eqref{bc}, then
\begin{equation}\label{wf}
E[t,\bar{x}^{\sigma}(t),\bar{x}^{\triangle}(t),q]\geq 0
\end{equation}
for all $t\in[t_{0}, t_{1}]^{\kappa}$ and all $q\in \R$, where we replace $\bar{x}^{\triangle}(t)$ by $\bar{x}^{\triangle}(t-)$ and $\bar{x}^{\triangle}(t+)$ at finitely many points $t$ where
$\bar{x}^{\triangle}(t)$ does not exist.
\end{Th}
\begin{proof}
Assume that  $\bar{x}$ is a strong local minimum for
\eqref{vp}--\eqref{bc}. We consider two cases. First, suppose that
$a\in [t_{0},t_{1}]^{\kappa}$ is a right-scattered point. If
$\bar{x}$ is a strong minimizer for the problem
\eqref{vp}--\eqref{bc}, then it the restriction of $\bar{x}$ to
$[a,\sigma (a)]\cap \mathbb{T}$ is a strong minimizer for the
problem (see \cite{dynamic})
\begin{equation*}
     \int_{a}^{\sigma(a)}f(t,x^{\sigma}(t),x^{\triangle}(t))\triangle
     t \longrightarrow \min
\end{equation*}
\begin{equation*}
    x(a)=\bar{x}(a),\,  \ x(\sigma(a))=\bar{x}(\sigma(a)).
\end{equation*}
We define the function $h:\R\rightarrow \R$ by $h(q)=
\int_{a}^{\sigma(a)}f(t,\bar{x}^{\sigma}(t),q)\triangle t$. Hence,
$h(q)=\mu(a)f(a,\bar{x}^{\sigma}(a),q)$. By assumption
\eqref{convex}, we have immediately that
\begin{equation*}
h(q)-h(\bar{x}^{\triangle}(a))
-(q-\bar{x}^{\triangle}(a))h'(\bar{x}^{\triangle}(a))\geq 0 \, .
\end{equation*}
This gives
\begin{equation*}
E[a,\bar{x}^{\sigma}(a),\bar{x}^{\triangle}(a),q]\geq 0.
\end{equation*}

Second, we suppose that $a\in [t_{0},t_{1}]^{\kappa}$, $a<t_{1}$, is
a right-dense point and $[a,b]\cap \T$ is an interval between two
successive points where $\bar{x}^{\triangle}(t)$ does not exist.
Then, there exists a sequence $\{\varepsilon_{k}:k \in
\N\}\subset[t_{0},t _{1}]$ with $\lim_{k \rightarrow
\infty}\varepsilon _{k}=a$. Let $\tau$ be any number such that
$\sigma (\tau) \in [a,b)$ and $q \in \R$. We define the function
$x:[t_{0},t_{1}]\cap \mathbb{T}\rightarrow \R$ as follows:
\begin{gather*}
x(t)=
\begin{cases}
\bar{x}(t) & \text{ if } t\in[t_{0},a]\cup [b,t_{1}] \\
 X(t) & \text{ if } t\in[a,\tau], \\
 \phi (t,\tau) & \text{ if } t\in [\tau,b] \, ,
\end{cases}
\end{gather*}
where
\begin{equation*}
X(t)=\bar{x}(a)+q(t-a), \quad q\in\R,
\end{equation*}
\begin{equation*}
    \phi
    (t,\tau)=\bar{x}(t)+\frac{X(\tau)-\bar{x}(\tau)}{b-\tau}(b-t).
\end{equation*}
Clearly, given $\delta >0$, for any $q$ one can choose $\tau$ such
that $\|x- \bar{x}\|_{0}<\delta$. Let us now consider the function
$K$ defined for all $\tau\in [a,b)\cap \T$ such that
$\sigma(\tau)\in [a,b)\cap \T$ with the values
$K(\tau)=\mathcal{L}[x]-\mathcal{L}[\bar{x}]$. Since
$\mathcal{L}[x]\geq \mathcal{L}[\bar{x}]$, by hypothesis,
$K(\tau)\geq 0$ and $K(a)=0$, it follows by Theorem~1.12 in
\cite{BP2} that $K^{\triangle}(a) \geq 0$. By the definition of $x$,
we have

\begin{multline*}
K(\tau)=\int_{a}^{\tau}\{f[t,X^{\sigma}(t),X^{\triangle}(t)]
-f[t,\bar{x}^{\sigma}(t),\bar{x}^{\triangle}(t)]\}\triangle
t\\
+\int_{\tau}^{b}\{f[t,\phi(\sigma(t),\tau),\phi^{\triangle_{1}}(t,\tau)]
-f[t,\bar{x}^{\sigma}(t),\bar{x}^{\triangle}(t)]\}\triangle t
\end{multline*}
so that, by Theorem 5.37 in \cite{BG} and Theorem~7.1 in
\cite{BGpartial}, we obtain
\begin{equation}\label{1}
K^{\triangle}(\tau) =f[\tau,X^{\sigma}(\tau),X^{\triangle}(\tau)]
-f[\tau,\phi(\sigma(\tau),\sigma(\tau)),
\phi^{\triangle_{t}}(\tau,\sigma(\tau))]
\end{equation}
\begin{equation}\label{2}
+\int_{\tau}^{b}\{f_{x}[t,\phi(\sigma(t),\tau),
\phi^{\triangle_{1}}(t,\tau)]\phi^{\triangle_{2}}(\sigma(t),\tau)
+f_{r}[t,\phi(\sigma(t),\tau),\phi^{\triangle_{1}}(t,
\tau)]\phi^{\triangle_{1}\triangle_{2}}(t, \tau)\}\triangle t \, .
\end{equation}
Invoking the relation $\phi^{\triangle_{1}\triangle_{2}}
=\phi^{\triangle_{2}\triangle_{1}}$ (see Theorem~6.1 in
\cite{BGpartial}), integration by parts gives
\begin{equation*}
\int_{\tau}^{b}f_{r}\phi^{\triangle_{2}\triangle_{1}}(t,\tau)\triangle
t=f_{r}\phi^{\triangle_{2}} (t,\tau)|_{\tau}^{b}
-\int_{\tau}^{b}f_{r}^{\triangle_{1}}\phi^{\triangle_{2}}(\sigma(t),\tau)\triangle
t.
\end{equation*}
Thus, \eqref{2} becomes
\begin{equation*}
\int_{\tau}^{b}[f_{x}
-f_{r}^{\triangle_{1}}]\phi^{\triangle_{2}}(\sigma(t),
\tau)\triangle t
\end{equation*}
\begin{equation}\label{5}
+f_{r}\phi^{\triangle_{2}} (t,\tau)|_{\tau}^{b}.
\end{equation}
From the definition of $\phi(t,\tau)$ we have
\begin{equation*}
\phi^{\triangle_{2}}(t,\tau)=\frac{(X^{\triangle}(\tau)-\bar{x}^{\triangle}(\tau))(b-\tau)+X(\tau)-\bar{x}(\tau)}{(b-\tau)(b-\sigma(\tau))}(b-t)
\end{equation*}
so that $\phi^{\triangle_{2}}(b,a)=0$,
$\phi^{\triangle_{2}}(a,a)=X^{\triangle}(a)-\bar{x}^{\triangle}(a)$.
Also, $\phi(\sigma(t),a)=\bar{x}(\sigma(t))$,
$\phi^{\triangle_{1}}(t,a)=\bar{x}^{\triangle}(t)$. Thus, letting
$\tau=a$ in  \eqref{5} we obtain
\begin{equation*}
-f_{r}[a,\bar{x}(\sigma(a)),\bar{x}^{\triangle}(a)][X^{\triangle}(a)-\bar{x}^{\triangle}(a)].
\end{equation*}
Since $\bar{x}$ verifies the Euler-Lagrange equation (see
\cite{b7}), we get
\begin{equation*}
\int_{\tau}^{b}\{f_{x}[t,\bar{x}(\sigma(t)),\bar{x}^{\triangle}(t)]
-f_{r}^{\triangle}[t,\bar{x}(\sigma(t)),\bar{x}^{\triangle}(t)]
\}\phi^{\triangle_{2}}(\sigma(t),\tau)\triangle t=0.
\end{equation*}
On account of the above, from \eqref{1}--\eqref{2} we have
\begin{multline*}
K^{\triangle}(a)=f[a,X^{\sigma}(a),X^{\triangle}(a)]-f[a,\phi(\sigma(a),\sigma(a)),\phi^{\triangle_{1}}(a,\sigma(a))]\\
-f_{r}[a,\bar{x}(\sigma(a)),\bar{x}^{\triangle}(a)][X^{\triangle}(a)-\bar{x}^{\triangle}(a)].
\end{multline*}
However, $X^{\sigma}(a)=\bar{x}^{\sigma}(a)$, $X^{\triangle}(a)=q$,
$\phi(\sigma(a),\sigma(a))=\bar{x}^{\sigma}(a)$,
$\phi^{\triangle_{1}}(a,\sigma(a))=\bar{x}^{\triangle}(a)$.
Therefore,
\begin{equation*}
K^{\triangle}(a)=f[a,\bar{x}^{\sigma}(a),q]
-f[a,\bar{x}^{\sigma}(a),\bar{x}^{\triangle}(a)]
-f_{r}[a,\bar{x}^{\sigma}(a),\bar{x}^{\triangle}(a)][q-\bar{x}^{\triangle}(a)],
\end{equation*}
and from this
\begin{equation*}
K^{\triangle}(a) =E[a,\bar{x}^{\sigma}(a),\bar{x}^{\triangle}(a),q]
\geq 0.
\end{equation*}
To establish the condition \eqref{wf} for all $t\in [t_{0},
t_{1}]^{\kappa}$, we consider the limit $t\rightarrow t_{1}$ from
left when $t_{1}$ is left-dense, and the limit $t\rightarrow t_{p}$
from left and from right when $t_{p}\in T$.
\end{proof}

\begin{Remark}
For $\mathbb{T}=\mathbb{R}$ problem \eqref{vp}--\eqref{bc} coincides
with the classical problem of the calculus of variations. Condition
\eqref{convex} is then trivially satisfied and Theorem \ref{Weier} is known as the Weierstrass necessary condition.
\end{Remark}

\begin{Remark}
Let $\mathbb{T}$ be a time scale with $\mu(t)$ depending
on $t$ and such that the time scale interval $[t_{0}, t_{1}]$ may be written as follows: $[t_{0}, t_{1}]=L\cup U$ with $\mu(t)\neq 0$ for all $t\in L$ and $\mu(t)= 0$ for all $t\in U$. An example
of such time scale is the Cantor set \cite{BP}. Then, for $t\in U$ the condition \eqref{convex} is trivially satisfied, while for $t\in L$ \eqref{convex} is nothing more than convexity of $f$
with respect to $r$.
\end{Remark}


\section{Special Cases}
\label{sec:ex}

Let $\mathbb{T}=\mathbb{Z}$. If $\bar{x}$ is a local minimum of the problem
\begin{equation*}
\begin{gathered}
 \text{minimize} \quad   \mathcal{L}[x]=\sum_{t=t_{0}}^{t_{1}-1}f(t,x(t+1),\triangle
 x(t)) \, ,\\
x(t_{0})=\alpha,\,  \ x(t_{1})=\beta, \quad \alpha,\beta\in \R,
\end{gathered}
\end{equation*}
and the function $f(t,x,r)$ is convex with respect to $r \in \R$ for each $(t,x)\in[t_{0}, t_{1}-1]\times \R$, then
$E[t,\bar{x}(t+1),{\triangle}\bar{x}(t),q]\geq 0$
for all $t\in[t_{0}, t_{1}-1]$ and all $q\in \R$.

\medskip

Let now $\mathbb{T}=q^\mathbb{N}$, $q>1$. If $\bar{x}$ is a local minimum of the problem
\begin{equation*}
\text{minimize} \quad \mathcal{L}[x]=\sum_{t\in[t_{0},t_{1})}(q-1)t
f \left(t,x(qt),\frac{x(qt)-x(t)}{qt-t}\right)\, ,
\end{equation*}
\begin{equation*}
x(t_{0})=\alpha,\,  \ x(t_{1})=\beta, \quad \alpha,\beta\in \R,
\end{equation*}
and the function $f(t,x,r)$ is convex with respect to $r \in \R$ for each $(t,x)\in[t_{0}, t_{1})\times \R$, then
\begin{equation*}
E \left[t,\bar{x}(qt),\frac{\bar{x}(qt)-\bar{x}(t)}{qt-t},p
\right]\geq 0
\end{equation*}
for all $t\in[t_{0}, t_{1})$ and all $p\in \R$.


\section*{Acknowledgments}

The first author was supported by Bia{\l}ystok Technical University,
via a project of the Polish Ministry of Science and Higher Education
"Wsparcie miedzynarodowej mobilnosci naukowcow"; the second author
by the R\&D unit CEOC, via FCT and the EC fund FEDER/POCI 2010. ABM
is grateful to the good working conditions at the Department of
Mathematics of the University of Aveiro where this research was
carried out.



\end{document}